\documentclass[11pt]{article}
\usepackage{amssymb}
\usepackage{}
\usepackage{pifont}
\usepackage{amscd}
\usepackage{amsfonts}
\usepackage{amsmath, amsthm}
\usepackage{amssymb, amsxtra}

\usepackage{graphicx}

\usepackage{multicol}
\usepackage{epstopdf}
\usepackage{epsf}
\usepackage{float}
\usepackage{extarrows}
\usepackage[a4paper, text={.8\paperwidth,.83\paperheight}, ratio=1:1]{geometry}
\usepackage[format=hang,font=small,textfont=it]{caption}
\usepackage{authblk}

\usepackage[colorlinks=true, linkcolor=blue, citecolor=blue]{hyperref}
\usepackage[english]{babel}
\usepackage{latexsym}

\usepackage{tikz}
\usetikzlibrary{calc}
\usepackage[tight]{subfigure}

\usepackage{color}
\definecolor{greenbean}{RGB}{199,237,204}
\newtheorem{thm}{Theorem}[]

\newtheorem{prop}[thm]{Proposition}
\newtheorem{cor}[thm]{Corollary}
\newtheorem{eg}{Example}[]

\newtheorem{Def}[thm]{Definition}

\def\C{{\mathbb C}}

\def\P{{\mathbb P}}

\def \ta{\tau}

\def \ta1{\tau_1}

\def \G{\Gamma}

\begin{document}
\title{On the  Galois covers of degenerations of surfaces of minimal degree
\footnotetext{\hspace{-1.8em} Email address: M. Amram: meiravt@sce.ac.il; C. Gong: cgong@suda.edu.cn; Jia-Li Mo:  mojiali0722@126.com;\\
2020 Mathematics Subject Classification. 05E15, 14J10, 14J25, 14N20.}}

\author[1]{Meirav Amram}
\author[2]{Cheng Gong}
\author[2]{Jia-Li Mo}

\affil[1]{\small{Shamoon College of Engineering, Ashdod, Israel}}
\affil[2]{\small{Department of Mathematics, Soochow University, Shizi RD 1, Suzhou 215006, Jiangsu, China}}

\date{}
\maketitle

\abstract{We investigate  the topological structures  of Galois covers of   surfaces  of minimal degree (i.e., degree $n$) in $\mathbb{CP}^{n+1}$. We prove that for $n\geq 5$, the Galois covers  of any surfaces  of minimal degree are simply-connected surfaces of general type.}

\section{Introduction}\label{outline}

The moduli space of surfaces is a hot topic that mathematicians refer to, see for example \cite{Hu1,Hu2}. The moduli space of surfaces of general type is a quasi-projective coarse moduli  scheme \cite{Gie}. Unlike  the moduli of curves, it is not irreducible. Catanese \cite{C1, C2} and Manetti \cite{Ma} characterized the structures and the number of components of some moduli spaces. Not much more was thereafter known about the moduli space of surfaces of general type. Then, in \cite{Tei2}, Teicher defined some new invariants of surfaces that were stable on connected components of moduli space. These new invariants come from the polycyclic structure of the fundamental group of the complement of the branch curve $S$, of the generic projection from a surface $X$, to $\mathbb{CP}^2$.

The fundamental group $\pi_1(\mathbb{CP}^2-S)$  of the complement of $S$ does not change when the complex structure of $X$ changes continuously. In fact, all surfaces in the same component of the moduli space have the same homotopy type and therefore have the same  group $\pi_1(\mathbb{CP}^2-S)$.

In \cite{MoTe87a} and \cite{MoTe87}, Moishezon-Teicher showed that if $X_{\text{Gal}}$ is the Galois cover of $X$, then its fundamental group $\pi_1(X_{\text{Gal}})$  can be obtained as a quotient group of $\pi_1(\mathbb{CP}^2-S)$. As a consequence, $\pi_1(X_{\text{Gal}})$ does not change when the complex structure of $X$ changes continuously. Based on this idea, they constructed  a series of simply-connected algebraic surfaces of general type with positive and zero indices, disproving the Bogomolov Conjecture, which states that an algebraic surface of general type with a positive index has an infinite fundamental group. These examples  have important value in  the geography of algebraic  surfaces.

To compute the group $\pi_1(X_{\text{Gal}})$, we construct some degenerations and work with special singularities that are defined and explained below.

In \cite{Zappa} and \cite{zg2} Zappa  first studied degenerations of scrolls to unions of planes. Then, in \cite{C-C-F-M-2} and \cite{C-C-F-M-4}, Calabri, Ciliberto, Flamini, and Miranda considered the flat degenerations of surfaces whose general fiber is a smooth projective algebraic surface and whose central fiber is a union of planes  in $\mathbb{CP}^r, r\geq3$ with the following singularities:

\begin{itemize}
 \item in codimension 1, double curves that are smooth and irreducible along with two surfaces meeting transversally;
\item  multiple points that are locally analytically isomorphic to the vertex of a cone over the stick curve, with an arithmetic genus of either 0 or 1, which is projectively normal in the projective space it spans.
\end{itemize}

These multiple points will be called {\it Zappatic singularities} and the  central fiber (a union of planes) will be called {\it a planar Zappatic surface}. A degeneration  is called  {\it a planar  Zappatic degeneration}, that is, a smooth surface that flatly degenerates to a planar Zappatic surface, as described in Section \ref{method} and Figure \ref{abc}.

The topological structure of a  Zappatic degeneration is complicated. In \cite{C-C-F-M-2},  \cite{C-C-F-M-3},  and \cite{C-C-F-M-4}, the authors discuss the effect of a  Zappatic degeneration  on its numerical invariants, such as the Euler-Poincar\'{e} characteristic, sectional genus, geometric genus,  and Chern numbers.

We are interested in the  topological structure of Galois covers of  planar Zappatic degenerations, which have not been as extensively known until now. We have discussed these  degenerations  in \cite{degree6,AGTX1,A-R-T}. In this paper, we continue to study the  planar Zappatic degenerations and find the group $\pi_1(X_{\text{Gal}})$.
We focus on special  planar Zappatic degenerations --- the degenerations to the cone over the stick curve $C_{R_k}$ (i.e., unions of lines with only nodes as singularities).
Now,   let $T_k$ be any connected tree with $k\geq3$ vertices. This corresponds to a non-degenerate stick curve of degree $k$ in ${\mathbb{CP}^k}$, which we denote by $C_{T_k}$. Moreover, when the tree $T_k$ consists of a chain $R_k$ of length $k$. The curve $C_{R_k}$ is the union of $k$ lines $l_1, l_2, \dots , l_k$ spanning ${\mathbb{CP}^k}$, such that $l_i \cap l_j = \emptyset$ iff $|i - j|>1$, as described in Figure \ref{R}.

\begin{figure}[H]
		\definecolor{circ_col}{rgb}{0,0,0}
\begin{center}
		\begin{tikzpicture}[x=1cm,y=1cm,scale=0.8]
\draw [fill=circ_col] (-7,0) circle (1pt);
		\draw [fill=circ_col] (-6,0)  circle (1pt);
		\draw [fill=circ_col] (-5,0) circle (1pt);
		\draw [fill=circ_col] (-4,0)  circle (1pt);
		\draw [fill=circ_col] (-3,0) circle (1pt);
		\draw [fill=circ_col] (-2,0)  circle (1pt);
		\draw [fill=circ_col] (-1,0)  circle (1pt);
        \draw  (-7,0) -- (-6,0) -- (-5,0) -- (-4,0) -- (-3,0) -- (-2,0) -- (-1,0);
      \draw  (2.5,0.5) -- (1.5,-0.5);
      \draw  (2,0.5) -- (3,-0.5);
      \draw  (3.5,0.5) -- (2.5,-0.5);
       \draw  (3,0.5) -- (4,-0.5);
       \draw  (4.5,0.5) -- (3.5,-0.5);
       \draw  (4,0.5) -- (5,-0.5);
        \draw  (5.5,0.5) -- (4.5,-0.5);
      \draw [fill=circ_col] (2.25,0.25)  circle (1pt);
      \draw [fill=circ_col] (3.25,0.25)  circle (1pt);
      \draw [fill=circ_col] (4.25,0.25)  circle (1pt);
\draw [fill=circ_col] (2.75,-0.25)  circle (1pt);
      \draw [fill=circ_col] (3.75,-0.25)  circle (1pt);
      \draw [fill=circ_col] (4.75,-0.25)  circle (1pt);
      \end{tikzpicture}
      \end{center}
\setlength{\abovecaptionskip}{0.15cm}
	\caption{A chain $R_k$ ~~~~~~~~~~~~~~~~~~~~ A stick curve $C_{R_k}$}\label{R}
\end{figure}

It is well known that the smallest possible degree of an irreducible, non-degenerate surface $X\subset \mathbb{CP}^{n+1}$ is $n$. Such surface is said to be a surface of minimal degree.
Because any surface $X$ of minimal degree in $\mathbb{CP}^{n+1}$ can be flatly degenerated to the cone over the stick curve $C_{T_n}$ (\cite[Corollary 12.2]{C-C-F-M-5}),  we get our main theorem:

\begin{thm}\label{mainthm}
  The Galois covers  of \underline{surfaces of minimal degree} in $\mathbb{CP}^{n+1}$ are simply-connected surfaces of general type for $n\geq 5$.
\end{thm}

The paper is organized as follows:  In Section \ref{method}, we explain the methods we use and the terminology related to the paper. In Section \ref{Rk-type}, we deal with the special case: degenerations to the cone over the stick curve  $C_{R_k}$ (see Theorem \ref{thm(n+1)point}) and prove that the related Galois covers are of general type.
In Section \ref{main}, we relate the general case on degeneration to  the cone over the stick curve  $C_{T_k}$ to the special case in Section \ref{Rk-type}, and prove Theorem \ref{mainthm}.

\paragraph{Acknowledgements:}
We thank Dr. Yi Gu for useful discussions about the degeneration of surfaces. This research was supported by the NSFC and  ISF-NSFC joint research program (Grant No. 2452/17). It was also partly supported by the Natural Science Foundation of Jiangsu Province~(BK 20181427).
We thank two anonymous referees for great comments and suggestions.

\section{Method and terminology}\label{method}

In this section, we describe  the methods, the fundamental background, and some terminology that we use  in this paper. We will use these methods and terminology in Section \ref{Rk-type}.  The reader can refer  to \cite{A-R-T} and \cite{BGT2} for more details.

We consider planar Zappatic surfaces and are interested in the Galois cover of each such surface. The fundamental group of the Galois cover is a significant invariant of the surface, as explained in the introduction, and we are going to calculate it.

To do this, we first need to understand the following setting:
Let $ X $ be a projective algebraic surface embedded in projective space $ \mathbb{CP}^n $, for some $n$. Consider a generic projection  $f:\mathbb{CP}^n\to\mathbb{CP}^2$. The restriction of $ f|_X $ is branched along a curve $ S\subset \mathbb{CP}^2 $. The branch curve $ S $ can tell a lot about $ X $, but it is difficult to describe it explicitly. To tackle this problem we consider degenerations of $ X $, defined as follows.
	
\begin{Def}
Let $\Delta$ be the unit disk, and $X, X'$ be algebraic surface. Assume that $f$ and $f'$ are projective projections, where $f: X\rightarrow \mathbb{CP}^2$,~ $f': X'\rightarrow \mathbb{CP}^2$. We say that $f$ is a projective degeneration of $f'$ if there exists a flat family $\pi: \mathfrak{X} \rightarrow \Delta$ (and where $\mathfrak{X}\subseteq\Delta\times\mathbb{CP}^n, n\geq3$ is a closed subscheme of relative dimension two), and a morphism $F:\mathfrak{X} \rightarrow \Delta \times\mathbb{CP}^2 $, such that $F$ composed with the first projective is $\pi$, and:
\begin{itemize}
  \item  $\pi^{-1}(0)\simeq X.$
  \item  There exists $0\neq p_0\in \Delta$ such that $\pi^{-1}(p_0)\simeq X'.$
  \item  The family $\mathfrak{X} - \pi^{-1}(0)\rightarrow \Delta - \{0\}$ is smooth.
  \item  Restricting to $\pi^{-1}(0), F\simeq \{0\}\times$ $f$ under the identification of $\pi^{-1}(0)$ with $X$.
  \item  Restricting to $\pi^{-1}(p_0), F\simeq \{p_0\}\times f'$ under the identification of $\pi^{-1}(p_0)$ with $X'$.
\end{itemize}
\end{Def}
	
	We construct a degeneration of $ X $ into $X_0$, as a sequence of \emph{partial degenerations} $X: =X_r \leadsto X_{r-1} \leadsto \cdots X_{r-i} \leadsto
	X_{r-(i+1)} \leadsto \cdots \leadsto X_0$. The degeneration $X_0$ is a union of planes, and each plane is projectively equivalent to $\mathbb{CP}^2$ (see \cite{19} for detail).
	Consider generic projections $\pi^{(i)} : X_{i} \rightarrow \mathbb{CP}^2$ with the branch curves $S_i$, for $0 \leq i \leq r$. Note that $S_{i-1}$ is a degeneration of $S_{i}$. Because $ X_0 $ is a union of planes, its projection $ S_0 $ is a line arrangement.

One of the principal tools we use is a reverse process of degeneration, and it is called \emph{regeneration}. Using this tool, which was described in  \cite{19} as regeneration rules, we can recover $ S_i $ from $ S_{i-1} $. Applying it multiple times, we can recover the original branch curve $ S $	 from the line arrangement $ S_0 $. In the following diagram, we illustrate this process.
	
	\[\begin{CD}
	X\subseteq \mathbb{CP}^n  @>\text{degeneration}>> X_0\subseteq \mathbb{CP}^n \\
	@V\text{generic~ projection}VV                      @VV\text{generic~ projection}V \\
	S\subset \mathbb{CP}^2 @<\phantom{regeneration}<\text{regeneration}< S_0\subset \mathbb{CP}^2
	\end{CD}\]
	
A line in $ S_0 $ regenerates to a conic.
	The resulting components of the partial regeneration are tangent to each other.
	To get a transversal intersection of components, we regenerate further, and this gives us three cusps for each tangency point (see \cite{BGT2, 19} for more details). Therefore, the regenerated branch curve $ S $ is a cuspidal curve with nodes and branch points. Local braids of such singularities are as follows:
	\begin{enumerate}
		\item for a branch point, $Z_{j \;j'}$ is a counterclockwise
		half-twist of $j$ and $j'$ along a path below the real
		axis,
		\item for nodes, $Z^2_{i, j \;j'}=Z_{i\;j}^2 \cdot Z_{i\;j'}^2$ and $Z^2_{i \;i', j \;j'}=Z_{i'\;j'}^2 \cdot Z_{i\;j'}^2 \cdot Z_{i'\;j}^2 \cdot Z_{i\;j}^2$,
		\item for cusps, $Z^3_{i, j \;j'}=Z^3_{i\; j} \cdot (Z^3_{i\; j})^{Z_{j\;j'}} \cdot (Z^3_{i \; j})^{Z^{-1}_{j \; j'}}$.
	\end{enumerate}
By the braid monodromy technique of Moishezon-Teicher, we derive the braids related to $S$ as  conjugations of the above local forms  (i.e., $a^b = b^{-1}ab$). The reader can learn more about this technique in \cite{BGT2}; in the paper we give the final braids that are computed by this technique, as the computations themselves are too long and tiring.
Note that in several places we use the notation $ \bar{a} $ where $ a $ is a braid, which means the same braid as $ a $ but above the real axis.

Denote $ G:=\pi_1(\mathbb{CP}^2-S) $ and its standard generators as $\G_1, \G_{1}', \dots, \G_{2m}, \G_{2m}'$. By the van Kampen Theorem \cite{vk} we can get a presentation of $G$ by means of generators $\{\G_{j}, \G_{j}'\}$ and relations of the types:
	\begin{enumerate}
		\item for a branch point, $ Z_{j\; j'} $ corresponds to the relation $\G_{j} = \G_{j}'$,
		\item for a node, $ Z_{i\; j}^2 $ corresponds to   $[\G_{i},\G_{j}]=\G_{i}\G_{j}\G_{i}^{-1}\G_{j}^{-1}=e$,
		\item for a cusp, $ Z_{i\; j}^3 $ corresponds to  $\langle\G_{i},\G_{j}\rangle=\G_{i}\G_{j}\G_{i}\G_{j}^{-1}\G_{i}^{-1}\G_{j}^{-1}=e$.
	\end{enumerate}
To get all the relations, we write the braids in a product and collect all the relations that correspond to the different factors. To each list of relations we add the projective relation $\prod\limits_{j=m}^1 \G_{j}'\G_{j}=e$. See \cite{BGT2, 19} for full treatment of the subject.

This method also enables us to compute the fundamental group of the Galois cover $X_{\text{Gal}}$ of $X$.
\begin{Def}
We consider the fibered product arising from a general projection $f: X \to \mathbb{CP}^2$ of degree $n$ as
$$X\times_{f}\cdots\times_{f}X=\{(x_1, \ldots , x_n)\in X^n|~f(x_1)=\cdots=f(x_n)\}.$$
Let the extended diagonal be
$$ \triangle=\{(x_1, \ldots , x_n)\in X^n|~x_i=x_j, for ~some ~i\neq j\}.$$
The closure
$\overline{X\times_{f}\cdots\times_{f}X-\triangle}$ is called the Galois cover w.r.t. the symmetric group $S_n$ and denoted by $X_{\text{Gal}}$.
\end{Def}

Then, there is an exact sequence
	\begin{equation}\label{M-T}
	0 \rightarrow \pi_1(X_{\text{Gal}}) \rightarrow G_1 \rightarrow S_n \rightarrow 0,
	\end{equation}
where $ G_1:=G/\langle {\Gamma_j}^2 , {\Gamma'_j}^2 \rangle $ and the  map $G_1 \rightarrow S_n$ is a surjection of $G_1$ onto the symmetric group $S_n$. This epimorphism takes the generators of $G_1$ to transpositions in the symmetric group $S_n$ according to the order of the edges in the degeneration. We thus obtain a presentation of the fundamental group $\pi_1(X_{\text{Gal}})$ of the Galois cover, as the kernel of this epimorphism. Then we simplify the relations to produce a canonical presentation that identifies with $\pi_1(X_{\text{Gal}})$, using the theory of Coxeter covers of the symmetric groups.
We use a proposition from \cite{MoTe87}, as follows:
\begin{prop}
If $$\frac{G_1}{\{\prod_{j=1}^k \G_{j}'\G_{j}\}}\cong S_n,$$ then $X_{\text{Gal}}$ is simply-connected.
\end{prop}

\section{Degeneration to the cone over $C_{R_k}$}\label{Rk-type}

In this section, we pay attention to a special planar Zappatic degeneration --- the degeneration to the cone over the stick curve $C_{R_k}$ in $\mathbb{CP}^{k+1}$.
It is clear that every plane arrangement can be represented by a triangulation
as long as no three planes meet in a line and no plane meets more than three other planes. In Figure \ref{abc}, we depict a  schematic representation of  $X_0$, which is a cone over the stick curve  $C_{R_k}$ in $\mathbb{CP}^{k+1}$. Each triangle corresponds to a plane $\mathbb{CP}^2$ and each  intersection of two triangles corresponds to a common edge between the two planes. The existence of such degeneration  can be found in Corollary \ref{cor}.

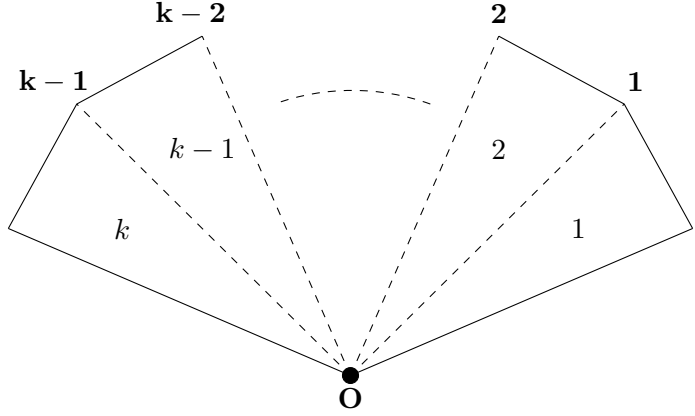
\begin{figure}[H]
	\begin{center}
		
			\definecolor{circ_col}{rgb}{0,0,0}
	\begin{tikzpicture}
	
	\begin{scope}[scale=1.5, xshift=200]
	
	\draw [fill=circ_col] (0,0) circle (2pt);
	\draw [-] (0,0) -- (3, 1.3);
	\draw [dashed] (0,0) -- (2.4, 2.4);
	\draw [dashed] (0,0) -- (1.3, 3);
	\draw [] (3, 1.3) -- (2.4, 2.4) -- (1.3, 3);
	\node[] at (2, 1.3) {1};
	\node[] at (1.3, 2) { 2};
	\node[] at (0, -0.2) {$\bf O$};
    \node[] at (2.5, 2.6) {$\bf 1$};
    \node[] at (1.3, 3.2) {$\bf 2$};
    \draw [-] (0,0) -- (-3, 1.3);
	\draw [dashed] (0,0) -- (-2.4, 2.4);
	\draw [dashed] (0,0) -- (-1.3, 3);
	\draw [-] (-3, 1.3) -- (-2.4, 2.4) -- (-1.3, 3);
	\node[] at (-2, 1.3) {$ k $};
	\node[] at (-1.3, 2) {$ k-1 $};
	 \node[] at (-2.6, 2.6) {$\bf k-1$};
    \node[] at (-1.4, 3.2) {$\bf k-2$};
	\draw [dashed] (0.7, 2.4) arc (70:110:2);
	
	\end{scope}
	
	\end{tikzpicture}
		
	\end{center}
	\setlength{\abovecaptionskip}{-0.15cm}
	\caption{A cone over $C_{R_k}$}\label{abc}
\end{figure}
\vspace{0.5cm}

We give now the following definition of an outer $(k-1)$-point, then we will explain how it relates to Figure \ref{abc}.

\begin{Def}\label{def:outer}
We call a $(k-1)$-point that is the intersection of $ k $ planes $ P_1,\dots,P_{k} $, where $ P_i $ intersects $ P_j $ in a line iff $ |i-j|=1 $, an \emph{outer $(k-1)$-point}. Especially noteworthy, a 1-point always comes from the intersection of 2 planes.
\end{Def}
	
Point $O$ in Figure \ref{abc} is  an outer $(k-1)$-point. We have also $k-1$ vertices that are 1-points. The branch curve $S_0$ is an arrangement of $k-1$ lines~(the dashed lines) that are the images of the $k-1$ edges through the generic  projection of $X_0$ onto $\mathbb{CP}^2$.

In Subsections \ref{6-type} and \ref{n+1-type}, we give the braids that are related to an  outer $5$-point and an outer $n$-point respectively. Then we can find the group $\pi_1(X_{\text{Gal}})$ and conclude the following theorem (see Theorem \ref{thm(n+1)point}):
{\em Let $\mathfrak{X}_k\rightarrow\Delta$ be a planar Zappatic degeneration, whose  central fiber $X_k$ is the cone over the stick curve $C_{R_k}$ in $\mathbb{CP}^{k+1}$ (for $k\geq 5$). Then the Galois cover of $X_k$ is a simply-connected surface of general type.}

In the following subsections, we follow the notations and formulations from \cite{degree6}. Before we continue to that part of the computations, we give some notations for simplicity and convenience, as follows:
we denote $\Gamma_j$ by $j$ and $\Gamma'_j$ by $j'$ in the group $G$; we use $B_{k}$ to denote the braid monodromy of an outer $k$-point; we write  $F_k$  instead of $\left(B_{k-1}\right)^{Z_{{(k-1)}\; {(k-1)}', k}^2}$, where $Z_{{(k-1)}\; {(k-1)}', k}^2$ is a full-twist of $k$ around
$k-1$ and $(k-1)'$; and we denote the following formula as $M_k$:
\begin{align*}
&M_k :=Z_{{(k-1)}\; {(k-1)}', k}^3 \cdot (Z_{1\; 1', k}^2)^{Z_{2\; 2', k}^{-2}\cdots Z_{{(k-2)}\; {(k-2)}', k}^{-2}} \cdot(Z_{2\; 2', k}^2)^{Z_{3\; 3', k}^{-2}\cdots Z_{{(k-2)}\; {(k-2)}', k}^{-2}}\\
&\cdot\cdot\cdot Z_{{(k-2)}\; {(k-2)}', k}^2\cdot{\bar{Z}}_{1\; 1', k'}^2 \cdot {\bar{Z}}_{2\; 2', k'}^2
\cdot\cdot\cdot{\bar{Z}}_{{(k-2)}\; {(k-2)}', k'}^2  \cdot (Z_{k\; k'})^{Z_{{(k-1)}\; {(k-1)}', k}^2},
& (k=1, 2 ,3 ,\cdots).
\end{align*}

\subsection{The cone over $C_{R_6}$}\label{6-type}

In \cite[Theorem 3.9]{AGTX} we have already considered the case of $k=5$. In order to help the reader better understand our proof, in this subsection we consider the case of $k=6$, see Figure \ref{5-point}.

\begin{figure}[H]
	\begin{center}
		
		\definecolor{circ_col}{rgb}{0,0,0}
		\begin{tikzpicture}[x=1cm,y=1cm,scale=1.8]
		
		\draw [fill=circ_col] (0,0) node [below] {$\bf O$} circle (1pt);
		\draw [fill=circ_col] (1,-1) node [anchor=west] {$\bf 7$} circle (1pt);
		\draw [fill=circ_col] (1,0) node [anchor=west] {$\bf 1$} circle (1pt);
		\draw [fill=circ_col] (1,1) node [anchor=west] {$\bf 2$} circle (1pt);
		\draw [fill=circ_col] (0,1) node [above] {$\bf 3$} circle (1pt);
		\draw [fill=circ_col] (-1,1) node [anchor=east] {$\bf 4$} circle (1pt);
		\draw [fill=circ_col] (-1,0) node [anchor=east] {$\bf 5$} circle (1pt);
		\draw [fill=circ_col] (-1,-1) node [anchor=east] {$\bf 6$} circle (1pt);
		
		\draw  (0,0) -- (1,-1) -- (1,0) -- (1,1) -- (0,1) -- (-1,1) -- (-1,0) -- (-1,-1) -- (0,0);
		
		\draw (0,0) -- node [below]{1} (1,0);
		\draw (0,0) -- node [above]{2} (1,1);
		\draw (0,0) -- node [anchor=east]{3} (0,1);
		\draw (0,0) -- node [above]{4} (-1,1);
		\draw (0,0) -- node [below]{5} (-1,0);

		\end{tikzpicture}
		
	\end{center}
	\setlength{\abovecaptionskip}{-0.15cm}
	\caption{A cone over  $C_{R_6}$}\label{5-point}
\end{figure}
\vspace{0.5cm}

\begin{prop}\label{thm6point}
 \rm  Let $\mathfrak{X}_6\rightarrow\Delta$ be a  planar Zappatic degeneration whose central fiber $X_6$ is the cone over the stick curve  $C_{R_6}$ in $\mathbb{CP}^{7}$. Then the Galois cover $X_{6,\text{Gal}}$ of  $X_6$ is a simply-connected surface.
\end{prop}

\begin{proof}
\setlength{\parindent}{1em}
The branch curve $S_0$ in $\mathbb{CP}^2$ is an arrangement of five lines, see Figure \ref{5-point}. We regenerate each vertex in turn  and compute the group $G_1$.

First, each of the vertices $i$ is an outer 1-point (for $i=1,\dots,5$) that regenerates to a conic; this gives rise to the braids $Z_{j\;j'}$ for $j=1,\dots,5$. We have the following relations in $G$ and also in $G_1$:
\begin{equation}\label{eq6.0}
1=1',\,2=2',\,3=3',\,4=4',\,5=5'.
\end{equation}
We will use the relations in (\ref{eq6.0}) as a prerequisite when we simplify relations (\ref{eq5.1})--(\ref{eq5.20}) in $G$.

Vertex $O$ is an outer $5$-point, and the related braids appear in $B_{5}$, as follows:
\begin{align*}
B_5=M_5 \cdot F_5
=M_5 \cdot (M_4)^{Z_{4\; 4', 5}^2} \cdot (B_3)^{Z_{3\; 3', 4}^2 Z_{4\; 4', 5}^2 },
\end{align*}
where
\begin{align*}
M_5= &Z_{4\; 4', 5}^3 \cdot (Z_{1\; 1', 5}^2)^{Z_{2\; 2', 5}^{-2} Z_{3\; 3', 5}^{-2}} \cdot
(Z_{2\; 2', 5}^2)^{Z_{3\; 3', 5}^{-2}} \cdot Z_{3\; 3', 5}^2 \cdot {\bar{Z}}_{1\; 1', 5'}^2 \cdot {\bar{Z}}_{2\; 2', 5'}^2 \cdot {\bar{Z}}_{3\; 3', 5'}^2  \cdot  (Z_{5\; 5'})^{Z_{4\; 4', 5}^2},\\
\end{align*}
and
\begin{align*}
F_5 = & \left(B_4\right)^{Z_{4\; 4', 5}^2}=(M_4)^{Z_{4\; 4', 5}^2}\cdot (F_4)^{Z_{4\; 4', 5}^2}\\
%=&\left((Z_{3\; 3', 4}^3)\cdot(Z_{1\; 1', 4}^2)^{Z_{2\; 2', 4}^{-2}}\cdot(Z_{2\; 2', 4}^2)\cdot({\bar{Z}}_{1\; 1', 4'}^2)\cdot({\bar{Z}}_{2\; 2', 4'}^2)\cdot(Z_{4\; 4'})^{Z_{3\; 3', 4}^2}\cdot F_4\right)^{Z_{4\; 4', 5}^2}\\
=&\left((Z_{3\; 3', 4}^3)\cdot(Z_{1\; 1', 4}^2)^{Z_{2\; 2', 4}^{-2}}\cdot(Z_{2\; 2', 4}^2)\cdot({\bar{Z}}_{1\; 1', 4'}^2)\cdot({\bar{Z}}_{2\; 2', 4'}^2)\cdot(Z_{4\; 4'})^{Z_{3\; 3', 4}^2}\right)^{Z_{4\; 4', 5}^2}\cdot (F_4)^{Z_{4\; 4', 5}^2}\\
=&(Z_{3\; 3', 4}^3)^{Z_{4\; 4', 5}^2}\cdot (Z_{1\; 1', 4}^2)^{Z_{2\; 2', 4}^{-2} Z_{4\; 4', 5}^2} \cdot (Z_{2\; 2', 4}^2)^{Z_{4\; 4', 5}^2} \cdot ({\bar{Z}}_{1\; 1', 4'}^2)^{Z_{4\; 4', 5}^2} \\
& \cdot({\bar{Z}}_{2\; 2', 4'}^2)^{Z_{4\; 4', 5}^2} \cdot  (Z_{4\; 4'})^{Z_{3\; 3', 4}^2 Z_{4\; 4', 5}^2} \cdot Z_{1', 2 \; 2'}^3 \cdot(Z_{1\; 1'})^{Z_{1', 2\; 2'}^2}  \\
& \cdot (Z_{2\; 2', 3}^3)^{Z_{1', 2\; 2'}^2 Z_{3\; 3', 4}^2 Z_{4\; 4', 5}^2} \cdot  (Z_{3\; 3'})^{Z_{2\; 2', 3}^2 Z_{1', 2\; 2'}^2 Z_{3\; 3', 4}^2 Z_{4\; 4', 5}^2} \cdot (Z_{1\; 1', 3\; 3'}^2)^{Z_{3\; 3' , 4}^2 Z_{4\; 4', 5}^2}.
\end{align*}
The braid $(Z_{5\; 5'})^{Z_{4\; 4', 5}^2}$ is depicted in the following picture:

\begin{figure}[ht]
\begin{center}
\scalebox{0.70}{\includegraphics{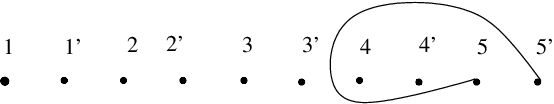}}
\end{center} \caption{The braid of $(Z_{5\; 5'})^{Z_{4\; 4', 5}^2}$}\label{7}
\end{figure}

The braids of $B_5$ give rise to three parts of relations in $G$. We will write down the
first part of the relations, which are the relations of braids of $M_5$:
\begin{equation}\label{eq5.1}
\langle 4,5\rangle=\langle 4',5\rangle=\langle 4^{-1}4'4,5\rangle=e,
\end{equation}
\begin{equation}\label{eq5.2}
[3'32'212^{-1}{2'}^{-1}3^{-1}{3'}^{-1},5]=[3'32'21'2^{-1}{2'}^{-1}3^{-1}{3'}^{-1},5]=e,
\end{equation}
\begin{equation}\label{eq5.3}
[3'323^{-1}{3'}^{-1},5]=[3'32'3^{-1}{3'}^{-1},5]=e,
\end{equation}
\begin{equation}\label{eq5.4}
[3,5]=[3',5]=e,
\end{equation}
\begin{equation}\label{eq5.5}
[4'43'32'212^{-1}{2'}^{-1}3^{-1}{3'}^{-1}4^{-1}{4'}^{-1},5^{-1}5'5]=
[4'43'32'21'2^{-1}{2'}^{-1}3^{-1}{3'}^{-1}4^{-1}{4'}^{-1},5^{-1}5'5]=e,
\end{equation}
\begin{equation}\label{eq5.6}
[4'43'323^{-1}{3'}^{-1}4^{-1}{4'}^{-1},5^{-1}5'5]=
[4'43'32'3^{-1}{3'}^{-1}4^{-1}{4'}^{-1},5^{-1}5'5]=e,
\end{equation}
\begin{equation}\label{eq5.7}
[4'434^{-1}{4'}^{-1},5^{-1}5'5]=[4'43'4^{-1}{4'}^{-1},5^{-1}5'5]=e,
\end{equation}
\begin{equation}\label{eq5.8}
5'=54'454^{-1}{4'}^{-1}5^{-1}.
\end{equation}
We simplify relations (\ref{eq5.1})--(\ref{eq5.8}), then get the following relations in $G_1$: $$\langle 4, 5\rangle=[1,5]=[2,5]=[3,5]=e.$$

Now we give the relations in $G$ of braids from $(M_4)^{Z_{4\; 4', 5}}$:
\begin{equation}\label{eq5.9}
\langle3,545^{-1}\rangle=\langle3',545^{-1}\rangle=\langle3^{-1}3'3,545^{-1}\rangle=e,
\end{equation}
\begin{equation}\label{eq5.10}
[2'212^{-1}{2'}^{-1},545^{-1}]=[2'21'2^{-1}{2'}^{-1},545^{-1}]=e,
\end{equation}
\begin{equation}\label{eq5.11}
[2,545^{-1}]=[2',545^{-1}]=e,
\end{equation}
\begin{equation}\label{eq5.12}
[3'32'212^{-1}{2'}^{-1}3^{-1}{3'}^{-1},54^{-1}4'45^{-1}]=[3'32'21'2^{-1}{2'}^{-1}3^{-1}{3'}^{-1},54^{-1}4'45^{-1}]=e,
\end{equation}
\begin{equation}\label{eq5.13}
[3'323^{-1}{3'}^{-1},54^{-1}4'45^{-1}]=[3'32'3^{-1}{3'}^{-1},54^{-1}4'45^{-1}]=e,
\end{equation}
\begin{equation}\label{eq5.14}
3^{-1}{3'}^{-1}54^{-1}4'45^{-1}3'3=545^{-1}.
\end{equation}
We simplify (\ref{eq5.9})--(\ref{eq5.14}), using  the  relations of $M_5$. We obtain the following relations in $G_1$:
$$\langle 3, 4\rangle=[1,4]=[2,4]=e.$$

We write down the relations in $G$ that are associated with the braids in $(B_3)^{Z_{3\; 3', 4}^2 Z_{4\; 4', 5}^2 }$; the elements of $G$ will appear with conjugations, according to the conjugation on $B_3$, as follows:
$$
\begin{aligned}
&3\rightarrow 434^{-1},~ 3'\rightarrow 43'4^{-1},~ 3^{-1}\rightarrow 43^{-1}4^{-1},~ 3'^{-1}\rightarrow 43'^{-1}4^{-1}; \\
&4\rightarrow 545^{-1},~ 4'\rightarrow 54'5^{-1},~ 4^{-1}\rightarrow 54^{-1}5^{-1},~ 4'^{-1}\rightarrow 54'^{-1}5^{-1}.
\end{aligned}
$$
We get in $G$ the relations:
\begin{equation}\label{eq5.15}
\langle 1',2\rangle=\langle 1',2'\rangle=\langle 1',2^{-1}2'2\rangle=e,
\end{equation}
\begin{equation}\label{eq5.16}
1=2'21'2^{-1}{2'}^{-1},
\end{equation}
\begin{equation}\label{eq5.17}
\begin{aligned}
&\langle 2'21'2{1'}^{-1}2^{-1}{2'}^{-1},545^{-1}354^{-1}5^{-1}\rangle=e,\\
&\langle 2'21'2'{1'}^{-1}2^{-1}{2'}^{-1},545^{-1}354^{-1}5^{-1}\rangle=e,\\
&\langle 2'21'2^{-1}2'2{1'}^{-1}2^{-1}{2'}^{-1},545^{-1}354^{-1}5^{-1}\rangle=e,
\end{aligned}
\end{equation}
\begin{equation}\label{eq5.18}
3=54^{-1}5^{-1}2'21'2^{-1}{2'}^{-1}{1'}^{-1}2^{-1}{2'}^{-1}545^{-1}3^{-1}3'3
54^{-1}5^{-1}2'21'2'2{1'}^{-1}2^{-1}{2'}^{-1}545^{-1},
\end{equation}
\begin{equation}\label{eq5.19}
\begin{aligned}
&[1,545^{-1}354^{-1}5^{-1}]=[1',545^{-1}354^{-1}5^{-1}]=e,\\
&[1,545^{-1}3'54^{-1}5^{-1}]=[1',545^{-1}3'54^{-1}5^{-1}]=e.
\end{aligned}
\end{equation}
We simplify (\ref{eq5.15})--(\ref{eq5.19}), using the ones from $M_5$ and $(M_4)^{Z_{4\; 4', 5}}$, and get the following relations in $G_1$:
$$\langle 1, 2\rangle=\langle 2, 3\rangle=[1,3]=e.$$

Moreover, the projective relation
\begin{equation}\label{eq5.20}
5'54'43'32'21'1=e,
\end{equation}
is trivial in $G_1$.

We summarize the relations in $G_1$, as  follows:
\begin{itemize}
\item[(1)] triple relations
\begin{equation}
\langle 1,2\rangle=\langle 2,3\rangle=\langle 3, 4\rangle=\langle 4, 5\rangle=e.
\end{equation}
\item[(2)] commutative relations
\begin{equation}
[1,3]=[1,4]=[2,4]=[1,5]=[2,5]=[3,5]=e.
\end{equation}
\end{itemize}

It is easy to see that $\{1,2,3,4,5\}$ are the generators of $G_1$.
These relations are the same as the relations in $S_6$, hence $G_1\cong S_6$. It follows that $\pi_1(X_{6,\text{Gal}})$ is trivial, and the Galois cover  of  $X_6$ is a simply-connected surface.

\end{proof}

In the next  subsection, we will prove the general theorem (Theorem \ref{thm(n+1)point}) by using the same method. The reader can discern and follow the inductive steps in the transition from Subsection  \ref{6-type} to Subsection \ref{n+1-type}.

\subsection{The cone over $C_{R_{n+1}}$}\label{n+1-type}
In this subsection, we study the fundamental group of the Galois cover of a  Zappatic degeneration   whose central fiber is the cone over the stick curve  $C_{R_k}$ in $\mathbb{CP}^{k+1}$.
In order to further clarify the expressions in the proof, we set $k=n+1$, see Figure \ref{n-points}.
We then have the following general theorem:
 \begin{thm}\label{thm(n+1)point}
Let $\mathfrak{X}_{n+1}\rightarrow\Delta$ be a planar Zappatic degenerations whose  central fibers $X_{n+1}$ is a cone over a stick curve  $C_{R_{n+1}}$ in $\mathbb{CP}^{n+2}$. Then the Galois cover  of  $X_{n+1}$ is simply-connected surfaces.
\end{thm}

\begin{figure}[H]
	\begin{center}
		
			\definecolor{circ_col}{rgb}{0,0,0}
	\begin{tikzpicture}
	
	\begin{scope}[scale=1.5, xshift=200]
	
	\draw [fill=circ_col] (0,0) circle (2pt);
	\draw [-] (0,0) -- (3, 1.3);
	\draw [-] (0,0) -- (2.4, 2.4);
	\draw [-] (0,0) -- (1.3, 3);
	\draw [-] (3, 1.3) -- (2.4, 2.4) -- (1.3, 3);
	\node[] at (2, 1.3) {1};
	\node[] at (1.3, 2) {2};
	\node[] at (0, -0.2) {$\bf O$};
    \node[] at (2.5, 2.5) {$\bf 1$};
    \node[] at (1.4, 3.12) {$\bf 2$};
     \node[] at (-2.55, 2.5) {$\bf n$};
    \node[] at (-1.37, 3.15) {$\bf n-1$};
	\draw [-] (0,0) -- (-3, 1.3);
	\draw [-] (0,0) -- (-2.4, 2.4);
	\draw [-] (0,0) -- (-1.3, 3);
	\draw [-] (-3, 1.3) -- (-2.4, 2.4) -- (-1.3, 3);
	\node[] at (-2, 1.3) {$ n+1 $};
	\node[] at (-1.3, 2) {$ n $};
	
	\draw [dashed] (0.7, 2.4) arc (70:110:2);
	
	\end{scope}
	
	\end{tikzpicture}
		
	\end{center}
	\setlength{\abovecaptionskip}{-0.15cm}
	\caption{A cone over $C_{R_{n+1}}$}\label{n-points}
\end{figure}
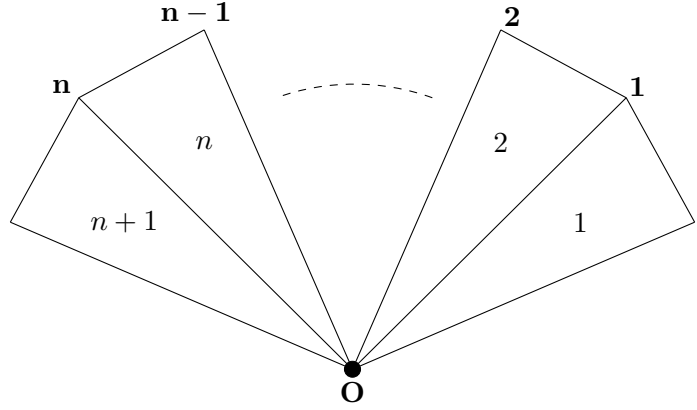
\vspace{0.5cm}

\begin{proof}

First, each of the vertices $i$ is an outer 1-point, for $i=1,\dots, n$, that regenerates to a conic, giving rise to the braids $Z_{j\;j'}$ for $j=1,\dots,n$. We have the following relations in $G$ and also in $G_1$:
\begin{equation}\label{eq8.0}
1=1', 2=2',\dots, n=n'.
\end{equation}
We will use the relations in (\ref{eq8.0}) as a prerequisite when we simplify the following $(n-2)$ parts of the
relations and the projective relation in $G$.

Vertex $O$ is an outer $n$-point, and the related braids are:
\begin{align*}
B_{n}=&M_n \cdot F_n\\
=&M_n \cdot (M_{n-1})^{Z_{{(n-1)}\; {(n-1)}', n}^2} \cdot (M_{n-2})^{Z_{{(n-2)}\; {(n-2)}', {(n-1)}}^2 Z_{{(n-1)}\; {(n-1)}', n}^2}\\
&\cdots (M_5)^{Z_{5\; 5', 6}^2 Z_{6\; 6', 7}^2 \cdots Z_{{(n-1)}\; {(n-1)}', n}^2}\cdot (M_4)^{Z_{4\; 4', 5}^2 Z_{5\; 5', 6}^2 \cdots Z_{{(n-1)}\; {(n-1)}', n}^2}\\
&\cdot (B_3)^{Z_{3\; 3', 4}^2 Z_{4\; 4', 5}^2 \cdots Z_{{(n-1)}\; {(n-1)}', n}^2}.
\end{align*}

In $G$, the braids of $B_{n}$ relate to $(n-2)$ parts of  relations.
The relations from braids in $M_n$ will be listed as the first part:
\begin{equation}\label{eq8.1}
\langle n-1,n\rangle=\langle (n-1)',n\rangle=\langle (n-1)^{-1}(n-1)'(n-1),n\rangle=e,
\end{equation}
\begin{equation}\label{eq8.2}
\begin{aligned}
&[(n-2)'(n-2)\cdots2'212^{-1}{2'}^{-1}\cdots(n-2)^{-1}{(n-2)'}^{-1},n]=e,\\
&[(n-2)'(n-2)\cdots2'21'2^{-1}{2'}^{-1}\cdots(n-2)^{-1}{(n-2)'}^{-1},n]=e,
\end{aligned}
\end{equation}
\begin{equation}\label{eq8.3}
\begin{aligned}
&[(n-2)'(n-2)\cdots3'323^{-1}{3'}^{-1}\cdots(n-2)^{-1}{(n-2)'}^{-1},n]=e,\\
&[(n-2)'(n-2)\cdots3'32'3^{-1}{3'}^{-1}\cdots(n-2)^{-1}{(n-2)'}^{-1},n]=e,
\end{aligned}
\end{equation}
\begin{align*}
&\cdots\cdots\cdots\cdots
\end{align*}
\begin{equation}\label{eq8.4}
\begin{aligned}
&[(n-2)'(n-2)(n-3)(n-2)^{-1}{(n-2)'}^{-1},n]=e,\\
&[(n-2)'(n-2)(n-3)'(n-2)^{-1}{(n-2)'}^{-1},n]=e,
\end{aligned}
\end{equation}
\begin{equation}\label{eq8.5}
[(n-2),n]=[(n-2)',n]=e,
\end{equation}
\begin{equation}\label{eq8.6}
\begin{aligned}
&[(n-1)'(n-1)\cdots2'212^{-1}{2'}^{-1}\cdots(n-1)^{-1}{(n-1)'}^{-1},n^{-1}n'n]=e,\\
&[(n-1)'(n-1)\cdots2'21'2^{-1}{2'}^{-1}\cdots(n-1)^{-1}{(n-1)'}^{-1},n^{-1}n'n]=e,
\end{aligned}
\end{equation}
\begin{equation}\label{eq8.7}
\begin{aligned}
&[(n-1)'(n-1)\cdots3'323^{-1}{3'}^{-1}\cdots(n-1)^{-1}{(n-1)'}^{-1},n^{-1}n'n]=e,\\
&[(n-1)'(n-1)\cdots3'32'3^{-1}{3'}^{-1}\cdots(n-1)^{-1}{(n-1)'}^{-1},n^{-1}n'n]=e,
\end{aligned}
\end{equation}
\begin{align*}
&\cdots\cdots\cdots\cdots
\end{align*}
\begin{equation}\label{eq8.8}
\begin{aligned}
&[(n-1)'(n-1)(n-2)(n-1)^{-1}{(n-1)'}^{-1},n^{-1}n'n]=e,\\
&[(n-1)'(n-1)(n-2)'(n-1)^{-1}{(n-1)'}^{-1},n^{-1}n'n]=e,
\end{aligned}
\end{equation}
\begin{equation}\label{eq8.9}
n'=n(n-1)'(n-1)n(n-1)^{-1}{(n-1)'}^{-1}n^{-1}.
\end{equation}
We simplify (\ref{eq8.1})--(\ref{eq8.9}), then get in $G_1$ the following relations:
\begin{itemize}
\item[(1)] triple relation
\begin{equation}
\langle n-1,n\rangle=e,
\end{equation}
\item[(2)] commutative relations
\begin{equation}
[1,n]=[2,n]=\cdots=[n-2,n]=e.
\end{equation}
\end{itemize}

The relations that are associated with the braids of $(M_{n-1})^{Z_{{(n-1)}\; {(n-1)}', n}^2}$ as the second part of the relations in $G$, are as follows:
\begin{equation}\label{eq8.10}
\begin{aligned}
&\langle n-2,n(n-1)n^{-1}\rangle=\langle (n-2)',n(n-1)n^{-1} \rangle=e,\\
&\langle (n-2)^{-1}(n-2)'(n-2),n(n-1)n^{-1}\rangle=e,
\end{aligned}
\end{equation}
\begin{equation}\label{eq8.11}
\begin{aligned}
&[(n-3)'(n-3)\cdots2'212^{-1}{2'}^{-1}\cdots(n-3)^{-1}{(n-3)'}^{-1},n(n-1)n^{-1}]=e,\\
&[(n-3)'(n-3)\cdots2'21'2^{-1}{2'}^{-1}\cdots(n-3)^{-1}{(n-3)'}^{-1},n(n-1)n^{-1}]=e,
\end{aligned}
\end{equation}
\begin{equation}\label{eq8.12}
\begin{aligned}
&[(n-3)'(n-3)\cdots3'323^{-1}{3'}^{-1}\cdots(n-3)^{-1}{(n-3)'}^{-1},n(n-1)n^{-1}]=e,\\
&[(n-3)'(n-3)\cdots3'32'3^{-1}{3'}^{-1}\cdots(n-3)^{-1}{(n-3)'}^{-1},n(n-1)n^{-1}]=e,
\end{aligned}
\end{equation}
\begin{align*}
&\cdots\cdots\cdots\cdots
\end{align*}
\begin{equation}\label{eq8.13}
\begin{aligned}
&[(n-3)'(n-3)(n-2)(n-3)^{-1}{(n-3)'}^{-1},n(n-1)n^{-1}]=e,\\
&[(n-3)'(n-3)(n-2)'(n-3)^{-1}{(n-3)'}^{-1},n(n-1)n^{-1}]=e,
\end{aligned}
\end{equation}
\begin{equation}\label{eq8.14}
[(n-3),n(n-1)n^{-1}]=[(n-3)',n(n-1)n^{-1}]=e,
\end{equation}
\begin{equation}\label{eq8.15}
\begin{aligned}
&[(n-2)'(n-2)\cdots2'212^{-1}{2'}^{-1}\cdots(n-2)^{-1}{(n-2)'}^{-1},n{(n-1)}^{-1}(n-1)'(n-1)n^{-1}]=e,\\
&[(n-2)'(n-2)\cdots2'21'2^{-1}{2'}^{-1}\cdots(n-2)^{-1}{(n-2)'}^{-1},n{(n-1)}^{-1}(n-1)'(n-1)n^{-1}]=e,
\end{aligned}
\end{equation}
\begin{equation}\label{eq8.16}
\begin{aligned}
&[(n-2)'(n-2)\cdots3'323^{-1}{3'}^{-1}\cdots(n-2)^{-1}{(n-2)'}^{-1},n{(n-1)}^{-1}(n-1)'(n-1)n^{-1}]=e,\\
&[(n-2)'(n-2)\cdots3'32'3^{-1}{3'}^{-1}\cdots(n-2)^{-1}{(n-2)'}^{-1},n{(n-1)}^{-1}(n-1)'(n-1)n^{-1}]=e,
\end{aligned}
\end{equation}
\begin{align*}
&\cdots\cdots\cdots\cdots
\end{align*}
\begin{equation}\label{eq8.17}
\begin{aligned}
&[(n-2)'(n-2)(n-3)(n-2)^{-1}{(n-2)'}^{-1},n{(n-1)}^{-1}(n-1)'(n-1)n^{-1}]=e,\\
&[(n-2)'(n-2)(n-3)'(n-2)^{-1}{(n-2)'}^{-1},n{(n-1)}^{-1}(n-1)'(n-1)n^{-1}]=e,
\end{aligned}
\end{equation}
\begin{equation}\label{eq8.18}
\begin{aligned}
{(n-2)}^{-1}{(n-2)'}^{-1}n{(n-1)}^{-1}(n-1)'(n-1)n^{-1}(n-2)'(n-2)=n(n-1)n^{-1}.
\end{aligned}
\end{equation}

Using the relations that are associated with $M_n$ to simplify (\ref{eq8.10})--(\ref{eq8.18}), we obtain the following relations in the group $G_1$:
\begin{itemize}
\item[(1)]triple relation
\begin{equation}
\langle n-2,n-1\rangle=e,
\end{equation}
%using $[n-2,n]=e$ from $(85)$.
\item[(2)] commutative relations
\begin{equation}
[1,n-1]=[2,n-1]=[n-3,n-1]=e.
\end{equation}
\end{itemize}

Similarly,  we can get the third part of the relations that relate to $(M_{n-2})^{Z_{{(n-2)}\; {(n-2)}', {(n-1)}}^2 Z_{{(n-1)}\; {(n-1)}', n}^2}$, then use the relations of $M_n$ and $(M_{n-1})^{Z_{{(n-1)}\; {(n-1)}', n}^2}$ to simplify them. We get the following relations in the group $G_1$:
\begin{itemize}
\item[(1)]triple relation
\begin{equation}
\langle n-3,n-2\rangle=e,
\end{equation}
\item[(2)] commutative relations
\begin{equation}
[1,n-2]=[2,n-2]=\cdots=[n-4,n-2]=e.
\end{equation}
\end{itemize}

Continuing this process, we can also get the 4th, 5th, $\dots, (n-3)$th parts of the relations and simplify them, then get the following relations in $G_1$:

\begin{itemize}
\item[(1)]triple relations
\begin{equation}
\langle n-4,n-3\rangle=\langle n-5,n-4\rangle=\cdots\cdots\cdots=\langle 3,4\rangle=e,
\end{equation}
\item[(2)] commutative relations
\begin{equation}
\begin{aligned}
&[1,n-3]=[2,n-3]=\cdots=[n-5,n-3]=e,\\
&[1,n-4]=[2,n-4]=\cdots=[n-6,n-4]=e,\\
&\cdots\cdots\cdots\\
&[1,4]=[2,4]=e.
\end{aligned}
\end{equation}
\end{itemize}

Finally, we write down the $(n-2)$th part of the relations in $G$, coming from the braids in $(B_3)^{Z_{3\; 3', 4}^2 Z_{4\; 4', 5}^2 \cdots Z_{{(n-1)}\; {(n-1)}', n}^2}$; this time they will appear with conjugated elements $(i=3, \dots, (n-1))$, as follows:
$$
i\rightarrow (i+1) i(i+1)^{-1}, i'\rightarrow (i+1) i'(i+1)^{-1}, i^{-1}\rightarrow (i+1) i^{-1}(i+1)^{-1}, i'^{-1}\rightarrow (i+1) i'^{-1}(i+1)^{-1}.
$$
We get the relations in $G$ as follows:

\begin{equation}\label{eq8.28}
\langle 1',2\rangle=\langle 1',2'\rangle=\langle 1',2^{-1}2'2\rangle=e,
\end{equation}
\begin{equation}\label{eq8.29}
1=2'21'2^{-1}{2'}^{-1},
\end{equation}
\begin{equation}\label{eq8.30}
\begin{aligned}
&\langle2'21'2{1'}^{-1}2^{-1}{2'}^{-1},n(n-1)n^{-1}(n-2)n(n-1)^{-1}n^{-1}(n-3)n(n-1)n^{-1}(n-2)^{-1}\\
&\cdots n(n-1)^{-1}n^{-1} 4 n(n-1)n^{-1}\cdots n(n-1)n^{-1} (n-2)^{-1}n(n-1)^{-1}n^{-1}3\\
& n(n-1)n^{-1}(n-2)n(n-1)^{-1}n^{-1}\cdots n(n-1)^{-1}n^{-1} 4^{-1}  n(n-1)n^{-1}\cdots\\
&(n-2)n(n-1)^{-1}n^{-1}(n-3)^{-1}n(n-1)n^{-1}(n-2)^{-1}n(n-1)^{-1}n^{-1}\rangle=e,
\end{aligned}
\end{equation}
\begin{equation}\label{eq8.31}
\begin{aligned}
&\langle2'21'2'{1'}^{-1}2^{-1}{2'}^{-1},n(n-1)n^{-1}(n-2)n(n-1)^{-1}n^{-1}(n-3)n(n-1)n^{-1}(n-2)^{-1}\\
&\cdots n(n-1)^{-1}n^{-1} 4 n(n-1)n^{-1}\cdots n(n-1)n^{-1} (n-2)^{-1}n(n-1)^{-1}n^{-1}3\\
&n(n-1)n^{-1}(n-2)n(n-1)^{-1}n^{-1}\cdots n(n-1)^{-1}n^{-1} 4^{-1}  n(n-1)n^{-1}\cdots\\
&(n-2)n(n-1)^{-1}n^{-1}(n-3)^{-1}n(n-1)n^{-1}(n-2)^{-1}n(n-1)^{-1}n^{-1}\rangle=e,
\end{aligned}
\end{equation}
\begin{equation}\label{eq8.32}
\begin{aligned}
&\langle2'21'2^{-1}2'2{1'}^{-1}2^{-1}{2'}^{-1},n(n-1)n^{-1}(n-2)n(n-1)^{-1}n^{-1}(n-3)n(n-1)n^{-1}(n-2)^{-1}\\
&\cdots n(n-1)^{-1}n^{-1} 4 n(n-1)n^{-1}\cdots n(n-1)n^{-1} (n-2)^{-1}n(n-1)^{-1}n^{-1}3\\
& n(n-1)n^{-1}(n-2)n(n-1)^{-1}n^{-1}\cdots n(n-1)^{-1}n^{-1} 4^{-1}  n(n-1)n^{-1}\cdots\\
&(n-2)n(n-1)^{-1}n^{-1}(n-3)^{-1}n(n-1)n^{-1}(n-2)^{-1}n(n-1)^{-1}n^{-1}\rangle=e,
\end{aligned}
\end{equation}
\begin{equation}\label{eq8.33}
\begin{aligned}
&3=n(n-1)n^{-1}(n-2)n(n-1)^{-1}n^{-1}\cdots n(n-1)n^{-1}4^{-1}n(n-1)^{-1}n^{-1}\cdots \\ &n(n-1)n^{-1}(n-2)^{-1}n(n-1)^{-1}n^{-1}2'21'2^{-1}{2'}^{-1}{1'}^{-1}2^{-1}{2'}^{-1}\\
&n(n-1)n^{-1}(n-2)n(n-1)^{-1}n^{-1}\cdots n(n-1)n^{-1}4n(n-1)^{-1}n^{-1}\cdots\\
&n(n-1)n^{-1}(n-2)^{-1}n(n-1)^{-1}n^{-1}3^{-1}3'3 n(n-1)n^{-1}(n-2)n(n-1)^{-1}n^{-1}\cdots\\
&n(n-1)n^{-1}4^{-1}n(n-1)^{-1}n^{-1}\cdots n(n-1)n^{-1}(n-2)^{-1}n(n-1)^{-1}n^{-1}\\
&2'21'2'2{1'}^{-1}2^{-1}{2'}^{-1}n(n-1)n^{-1}(n-2)n(n-1)^{-1}n^{-1}\cdots\\
&n(n-1)n^{-1}4n(n-1)^{-1}n^{-1}\cdots n(n-1)n^{-1}(n-2)^{-1}n(n-1)^{-1}n^{-1},
\end{aligned}
\end{equation}
\begin{equation}\label{eq8.34}
\begin{aligned}
&[1,n(n-1)n^{-1}(n-2)n(n-1)^{-1}n^{-1}(n-3)n(n-1)n^{-1}(n-2)^{-1}\\
&\cdots n(n-1)^{-1}n^{-1} 4 n(n-1)n^{-1}\cdots n(n-1)n^{-1} (n-2)^{-1}n(n-1)^{-1}n^{-1}3\\
& n(n-1)n^{-1}(n-2)n(n-1)^{-1}n^{-1}\cdots n(n-1)^{-1}n^{-1} 4^{-1}  n(n-1)n^{-1}\cdots\\
&(n-2)n(n-1)^{-1}n^{-1}(n-3)^{-1}n(n-1)n^{-1}(n-2)^{-1}n(n-1)^{-1}n^{-1}]=e,\\
&[1',n(n-1)n^{-1}(n-2)n(n-1)^{-1}n^{-1}(n-3)n(n-1)n^{-1}(n-2)^{-1}\\
&\cdots n(n-1)^{-1}n^{-1} 4 n(n-1)n^{-1}\cdots n(n-1)n^{-1} (n-2)^{-1}n(n-1)^{-1}n^{-1}3\\
& n(n-1)n^{-1}(n-2)n(n-1)^{-1}n^{-1}\cdots n(n-1)^{-1}n^{-1} 4^{-1}  n(n-1)n^{-1}\cdots\\
&(n-2)n(n-1)^{-1}n^{-1}(n-3)^{-1}n(n-1)n^{-1}(n-2)^{-1}n(n-1)^{-1}n^{-1}]=e,\\
&[1,n(n-1)n^{-1}(n-2)n(n-1)^{-1}n^{-1}(n-3)n(n-1)n^{-1}(n-2)^{-1}\\
&\cdots n(n-1)^{-1}n^{-1} 4 n(n-1)n^{-1}\cdots n(n-1)n^{-1} (n-2)^{-1}n(n-1)^{-1}n^{-1}3'\\
& n(n-1)n^{-1}(n-2)n(n-1)^{-1}n^{-1}\cdots n(n-1)^{-1}n^{-1} 4^{-1}  n(n-1)n^{-1}\cdots\\
&(n-2)n(n-1)^{-1}n^{-1}(n-3)^{-1}n(n-1)n^{-1}(n-2)^{-1}n(n-1)^{-1}n^{-1}]=e,\\
&[1',n(n-1)n^{-1}(n-2)n(n-1)^{-1}n^{-1}(n-3)n(n-1)n^{-1}(n-2)^{-1}\\
&\cdots n(n-1)^{-1}n^{-1}4 n(n-1)n^{-1}\cdots n(n-1)n^{-1} (n-2)^{-1}n(n-1)^{-1}n^{-1}3'\\
& n(n-1)n^{-1}(n-2)n(n-1)^{-1}n^{-1}\cdots n(n-1)^{-1}n^{-1} 4^{-1}  n(n-1)n^{-1}\cdots\\
&(n-2)n(n-1)^{-1}n^{-1}(n-3)^{-1}n(n-1)n^{-1}(n-2)^{-1}n(n-1)^{-1}n^{-1}]=e.\\
\end{aligned}
\end{equation}
We can use all the previous relations to simplify (\ref{eq8.28})--(\ref{eq8.34}), and  get the following relations in $G_1$:

\begin{itemize}
\item[(1)]triple relations
\begin{equation}
\langle1,2\rangle=\langle2,3\rangle=e,
\end{equation}
\item[(2)] commutative relation
\begin{equation}
[1,3]=e.
\end{equation}
\end{itemize}

We also have the projective relation:
\begin{equation}\label{eq8.35}
n'n(n-1)'(n-1)\cdots3'32'21'1=e,
\end{equation}
which is trivial in $G_1$.

In conclusion, we consider now all the above simplified relations in the group $G_1$:
\begin{itemize}
\item[(1)] triple relations
\begin{equation}
\langle 1,2\rangle=\langle 2,3\rangle=\cdots=\langle n-1,n\rangle=e,
\end{equation}
\item[(2)] commutative relations
\begin{equation}
\begin{aligned}
&[1,3]=[1,4]=\cdots=[1,n]=e,\\
&[2,4]=[2,5]=\cdots=[2,n]=e,\\
&\cdots\cdots\cdots\\
&[n-3,n-1]=[n-3,n]=e,\\
&[n-2,n]=e.
\end{aligned}
\end{equation}
\end{itemize}

It is easy to see that $\{1,2,\dots,n\}$ are the generators of $G_1$.
These relations are the same as the relations in $S_{n+1}$. Hence $G_1\cong S_{n+1}$. It is obvious that $\pi_1(X_{n+1,\text{Gal}})$ is trivial, and the Galois cover $X_{n+1,\text{Gal}}$ of $X_{n+1}$ is simply-connected.

\end{proof}

\subsection{General type}\label{General type}

When considering an algebraic surface $X$ as a topological 4-manifold, it has the Chern
numbers $c_1^2(X), c_2(X)$ as topological invariants.
In this subsection, we will prove that the Galois covers of the surfaces in Subsection \ref{n+1-type}  are surfaces of general type by using $c_1^2(X)$.
As a first step, we compute the Chern numbers $c_1^2(X)$. The formula
was treated in \cite[Proposition 0.2]{MoTe87} (proof there is given by F. Catanese).

\begin{prop}\label{chern}
Let $S$ be the branch curve of an algebraic surface $X$. Denote the
degree of the generic projection by $d$, ${\rm deg} S= m$.
Then,
$$c_1^2(X_\text{Gal})=\frac{d!}{4}(m-6)^2.$$
\end{prop}

Note that in Subsection \ref{n+1-type},  $d = n+1$ and $m = 2n$. Then by Proposition \ref{chern}, we obtain

\begin{itemize}
  \item $c_1^2(X_{5,\text{Gal}})=5!\cdot1$;
  \item $c_1^2(X_{6,\text{Gal}})=6!\cdot2^2$;
  \item $\cdots\cdots\cdots$
  \item $c_1^2(X_{n,\text{Gal}})=(n)!\cdot{(n-4)}^2$;
  \item $c_1^2(X_{n+1,\text{Gal}})=(n+1)!\cdot{(n-3)}^2$.
\end{itemize}
It is obvious that $c_1^2(X_{n,\text{Gal}})>0$ for $n \geq 5$. It means that the Galois covers are general type surfaces, as explained in \cite[Proposition X.1]{B} or \cite[Theorem 1.1]{BHPV}.

\section{Proof of Theorem \ref{mainthm}}\label{main}
In this section we prove Theorem \ref{mainthm}. First, we recall the following result of Pinkham:
\begin{thm}\label{Pin}
(\cite{Pi}) Let $X \subset \mathbb{CP}^n$ be a smooth, irreducible, and projectively
Cohen-Macaulay surface. Then $X$ degenerates to the cone over a hyperplane section of $X$.
\end{thm}

Let $C$ be the hyperplane section of $X$. Suppose that $C$
can be degenerated to a stick curve $C_0$. In this case, $S$
can be degenerated to the cone over the stick curve $C_0$.  Therefore:

\begin{cor}\label{cor}
(\cite[Corollary 12.2]{C-C-F-M-5}) Any surface $X$ of minimal degree (i.e., of degree $n$) in $\mathbb{CP}^{n+1}$
can be degenerated to the cone over the stick curve $C_{T_n}$, for any tree $T_n$ with $n$ vertices.
\end{cor}

Every nondegenerate irreducible surface of degree $n~(n \geq 5)$ in $\mathbb{CP}^{n+1}$ is a rational normal scroll. Any hyperplane section of such a surfaces is a normal rational curve. Beginning at a general point $p_i$ on each component of $C_{T_n}$, the line bundle $\mathcal {O}_{C_{T_{n}} (p_1 +\ldots+p_n)}$ is very ample. $C_{T_n}$ has arithmetic
genus $0$ and is a flat limit of rational normal curves in $\mathbb{CP}^n$. $C_{R_n}$
is a flat limit of rational normal curves (including  $C_{T_n}$) in $\mathbb{CP}^n$.
According to Corollary \ref{cor},  any surface $X$ of minimal degree in $\mathbb{CP}^{n+1}$ can be degenerated  to the cone over the stick curve $C_{R_n}$.

The fundamental group of the Galois cover $\pi_1(X_{\text{Gal}})$ does not change when the complex structure of $X$ changes continuously. In the previous narrative, we proved that any surface $X$ of minimal degree in $\mathbb{CP}^{n+1}$
can be degenerated to the cone over the stick curve $C_{R_n}$, so we can use Theorem \ref{thm(n+1)point} to get Theorem \ref{mainthm}.

\end{document}